\documentclass[11pt,sumlimits,nointlimits,namelimits,reqno]{amsart}

\usepackage{amsopn, amsmath, amssymb, amsthm,amscd, amsrefs}

\usepackage{cite}

\usepackage{enumitem}

\usepackage{graphicx,tabularx}
\usepackage[dvipsnames]{xcolor}
\usepackage{tikz}

\usepackage{xstring}

\usepackage[mathscr]{eucal}

\usepackage{mathtools, contour} 
\usepackage[normalem]{ulem} 

\usepackage{bbm, bm}

\newtheorem{theorem}{Theorem}[section]
\newtheorem{proposition}[theorem]{Proposition}

\newtheorem{corollary}[theorem]{Corollary}
\newtheorem{lemma}[theorem]{Lemma}

\newtheorem{remark}{Remark}
\newtheorem{question}{Question}

\theoremstyle{definition}
\newtheorem{definition}{Definition}

\numberwithin{equation}{section}

\newcommand{\hfspace}{\hspace{0.25cm}}

\newcommand{\N}{\mathbb{N}}
\newcommand{\Z}{\mathbb{Z}}
\newcommand{\Q}{\mathbb{Q}}
\newcommand{\R}{\mathbb{R}}

\newcommand{\Acal}{{\mathcal A}}
\newcommand{\Bcal}{{\mathcal B}}
\newcommand{\Ccal}{{\mathcal C}}

\newcommand{\Gcal}{{\mathcal G}}
\newcommand{\Hcal}{{\mathcal H}}

\newcommand{\Lcal}{{\mathcal L}}

\newcommand{\Pcal}{{\mathcal P}}

\newcommand{\Area}{\mathrm{area}}
 
\newcommand{\bfx}{{\bf x}} 
\newcommand{\bfy}{{\bf y}}

\newcommand{\bfc}{{\bf c}}
\newcommand{\bfh}{{\bf h}}

\newcommand{\floor}[1]{\lfloor #1 \rfloor }
\newcommand{\fracpart}[1]{\left\{ #1\right\}}

\newcommand{\dens}{\mathrm{dens}}
\newcommand{\densityofnwith}[1]{\dens\big(\{n\in\N:#1\}\big)}
\newcommand{\upperdens}{\overline{\dens}}
\newcommand{\lowerdens}{\underline{\dens}}

\newcommand{\blueurl}[1]{{ \color{blue}\url{#1} }}

\begin{document}

\title [On the greatest common divisor of $n,\floor{\alpha_1n},\floor{\alpha_2n^2},...,\floor{\alpha_kn^k}$]
{On the greatest common divisor of $n,\floor{\alpha_1n},\floor{\alpha_2n^2},...,\floor{\alpha_kn^k}$}

\author{J\'er\'emy Champagne}
\address{J. Champagne, Department of Mathematics and Statistics,
University of Ottawa,
150 Louis-Pasteur Private,
Ottawa, ON, K1N 9B4, Canada }
\email{jchampa3@uottawa.ca}


\begin{abstract}
    We answer a question of Bergelson and Richter about the probability of the relation $\gcd(n,\floor{\alpha_1n},\floor{\alpha_2n^2},...,\floor{\alpha_kn^k})=1$ for $n\in\N$ when $\alpha_1,...,\alpha_k$ are fixed irrational numbers. In particular, we avoid the use of exponential sums, as they are difficult to control when one of $\alpha_1,...,\alpha_k$ admits very efficient rational approximations. Instead, we use an elementary method similar to that of Erd\H{o}s and Lorentz, together with some general bounds on those $n$'s which share a large prime divisor with $\floor{\alpha_1n}$.
\end{abstract}

   \maketitle
\section{Introduction}\label{sec1}

It is a well-known fact, often attributed to Chebyshev, that the probability of a random pair $(m,n)\in\N^2$ being coprime tends towards
$$\prod_{p\textnormal{ prime}}\left(1-\frac{1}{p^2}\right)=\frac{1}{\zeta(2)}=\frac{6}{\pi^2}\approx 60.8\%,$$ 
where $\zeta$ denotes the Zeta-function of Riemann. To be more precise, we mean that the ratio
\begin{equation*}\label{S1eq:limchebyshev}
    \frac{|\{(m,n)\in\{1,...,M\}\times\{1,...,N\}:\gcd(m,n)=1\}|}{MN}
\end{equation*}
converges to $\zeta(2)^{-1}$ as $\min\{M,N\}\to\infty$.\\ 

Recall that any real number $x$ can be written as $x=\floor{x}+\{x\}$, where $\floor{x}\in\Z$ denotes its integer part and $\{x\}\in[0,1)$ denotes its fractional part. Watson \cite{WATSONgcdnalphan} showed in 1953 that, given an irrational number $\alpha$, the probability of the relation $\gcd(n,\floor{\alpha n})=1$ for a random $n\in\N$ also tends to $\zeta(2)^{-1}$. In other words, the Beatty sequence $\floor{\alpha n}$ exhibits a similar behaviour to that of a random integer chosen independently from $n$. Recall that the \textit{density} of a subset $\Acal\subseteq\N$ is given by the limit
$$\dens\,\Acal:=\lim_{N\to\infty}\frac{|\Acal\cap\{1,...,N\}|}{N},$$
which may or may not exist. The result of Watson can be stated more precisely as
\begin{equation}\label{S1eq:Watsondensity}
    \densityofnwith{\gcd(n,\floor{\alpha n})=1}=\frac{1}{\zeta(2)}.
\end{equation}
As made explicit in the later work of Spilker \cite{SPILKERalmostperiodicwatson}, the fact that the sequence $(\alpha n)_{n\in\N}$ is equidistributed mod 1 can also be used to obtain \eqref{S1eq:Watsondensity}. Recall that a sequence $(x_n)_{n\in\N}$ in $\R$ is equidistributed mod 1 if
$$\densityofnwith{a\leq\{x_n\}<b}=b-a$$
holds for all subinterval $[a,b)\subseteq[0,1)$.\\

This initial work of Watson prompted several authors to prove that, for different functions $f:[1,\infty)\to\R$, the relation $\gcd(n,\floor{f(n)})=1$ occurs for a set of $n\in\N$ whose density is also $\zeta(2)^{-1}$. In particular, Erd\H{o}s and Lorentz \cite{ELprobathatngncoprime} showed in 1959 that this density holds whenever $f(x)$ satisfies the following conditions.
\begin{enumerate}[label=EL\arabic*\textnormal{)}]
    \item $f$ is monotone increasing, positive, and differentiable with a piecewise continuous derivative,
    \item $\big(f(n)\big)_{n\in\N}$ is \textit{homogeneously} equidistributed mod 1, in the sense that $\big(\frac{1}{D}f(Dn)\big)_{n\in\N}$ is equidistributed mod 1 for all $D\in\N$,\medskip
    \item $\displaystyle\lim_{x\to\infty}\frac{f(x)\log\log x}{x}=0$,\medskip
    \item $\displaystyle\lim_{x\to\infty}\frac{xf'(x)}{\log\log\log x}=\infty$,\medskip
    \item There exists a constant $M>0$ such that $f'(y)\leq Mf'(x)$ for all $y\geq x>0$.
\end{enumerate}
Some typical examples of functions satisfying those five conditions would be those of the form $f(x)=x^c$ with $c\in(0,1)$, or $f(x)=\log(x)^r$ with $r>1$. It is worth noting that Delmer and Deshouillers \cite{DDpiatetskishapiro} were also able to prove
$$\densityofnwith{\gcd(n,\floor{n^c})=1}=\frac{1}{\zeta(2)}$$
for any parameter $c\in(1,+\infty)\setminus\N$.\\

In 2017, Bergelson and Richter \cite{BRgcdintparthardyfield} proved that $\gcd(n,\floor{f(n)})=1$ also occurs for a set of $n\in\N$ of density $\zeta(2)^{-1}$ whenever $f:[1,\infty)\to\R$ satisfies the conditions
\begin{enumerate}[label=BR\arabic*\textnormal{)}]
    \item $f$ belongs to a Hardy field\footnote{Hardy fields are fields of real-valued functions (up to equivalence given by their germ at $\infty$) which are closed under differentiation. For more information on Hardy fields and their role in the theory of equidistributed sequences, see \cite{BOSHERNITZANunifdirstrhardyfields}.} $\Hcal$,\medskip
    \item $\displaystyle \lim_{x\to\infty}\frac{f(x)}{\log (x)\log\log\log\log (x)}=\infty$,\medskip
    \item There exists $s\in\N$ such that $\displaystyle \lim_{x\to\infty}\frac{f(x)}{x^{s-1}}=\infty$ and $\displaystyle \lim_{x\to\infty}\frac{f(x)}{x^s}=0$.
\end{enumerate}
Their article also contains results involving multiple real-valued functions $f_1,....,f_k$ belonging to a same Hardy field $\Hcal$, who each satisfy conditions BR2 and BR3 individually, as well as
\begin{equation}
    \lim_{x\to\infty}\frac{f_{j+1}(x)}{f_j(x)(\log\log x)^2}=\infty
\end{equation}
for $j=1,...,k-1$. Under these hypotheses, Bergelson and Richter show that
\begin{equation*}
    \densityofnwith{\gcd(n,\floor{f_1(n)},...,\floor{f_k(n)})=1}=\frac{1}{\zeta(k+1)},
\end{equation*}
where indeed $\zeta(k+1)^{-1}$ corresponds to the probability of $\gcd(n_1,...,n_{k+1})=1$ for random integers $n_1,...,n_{k+1}\in\N$ chosen independently.\\

It is important to note that condition BR3 excludes the case where one of $f_1(x),...,f_k(x)$ is a polynomial with respect to $x$. With this constraint in mind, Bergelson and Richter formulated the following question (Question 1 from \cite{BRgcdintparthardyfield}).

\begin{question}[Bergelson and Richter, 2017]\label{S1q:BRmainquestion}
    Let $\alpha_1,\alpha_2,...,\alpha_k\in\R$ be irrational. Is it true that
    \begin{equation*}
        \densityofnwith{\gcd(n,\floor{\alpha_1n},\floor{\alpha_2n^2}...,\floor{\alpha_kn^k})=1}=\frac{1}{\zeta(k+1)}\,?
    \end{equation*}
\end{question}
Question \ref{S1q:BRmainquestion} was answered partially in a manuscript of Banks and Shparlinski \cite{BSgcdintpartpoly} using the Erd\H{o}s-Tur\'an inequality together with bounds on Weyl sums. However, this method faces significant difficulties in the case where at least one of $\alpha_1,...,\alpha_k\in\R$ is of Liouville-type, as it weakens the bounds on the corresponding Weyl sums. Recall that the \textit{Diophantine-type} of a number $\alpha\in\R$ is the supremum $\tau(\alpha)\in[1,+\infty]$ of all $\tau\geq 1$ such that the inequality $|\alpha-a/q|<q^{-\tau}$ has infinitely many solutions $(a,q)\in\Z\times\N$. Liouville-type numbers then correspond to those $\alpha\in\R$ with $\tau(\alpha)=\infty$.\\

In this article, we offer a complete and positive answer to Question \ref{S1q:BRmainquestion}. The method that we use is similar to that of Erd\H{o}s and Lorentz, but it also takes inspiration in the alternative proofs of Watson's result obtained by Estermann \cite{ESTERMANNprimitivelatticepoints} and Spilker \cite{SPILKERalmostperiodicwatson} respectively. In particular, we rely heavily on the equidistribution of the sequence $(\alpha_1n,\alpha_2n^2,...,\alpha_kn^k)\in\R^k$ mod 1. Recall that a sequence $(\bfx_n)_{n\in\N}$ in $\R^k$ is equidistributed mod 1 if, for any box $\Bcal=[a_1,b_1)\times\cdots\times[a_k,b_k)\subseteq[0,1)^k$, we have
\begin{equation*}
    \dens\big(\{n\in\N:\{\bfx_n\}\in\Bcal\}\big)=(b_1-a_1)\cdots(b_k-a_k),
\end{equation*}
where $\{\bfx\}$ is understood as $\big(\{x_1\},...,\{x_k\}\big)\in[0,1)^k$ for any $\bfx=(x_1,...,x_k)\in\R^k$. To simplify our proofs, we use a slightly more restrictive definition than the homogeneous equidistribution introduced by Erd\H{o}s and Lorentz.

\begin{definition}
    Let $k\in\N$. We say that a sequence $(\bfx_n)_{n\in\N}$ in $\R^k$ is \textit{totally-homogeneously} equidistributed mod 1 if, for any $D,b\in\N$, the sequence $\big(\frac{1}{D}\bfx_{Dn+b}\big)_{n\in\N}$ is equidistributed mod 1.
\end{definition}

The core of our argument applies to any sequence which is totally-homogeneously equidistributed mod 1, except for some convergence considerations where we must leverage the presence of the Beatty sequence $\floor{\alpha_1n}$. This leads to the following general result.

\begin{theorem}\label{S1thm:totallyhomog}
    Let $f_1,...,f_k:\N\to\R$ be functions such that
    \begin{enumerate}
        \item the sequence in $\R^k$ given by $\bfx_n:=\big(f_1(n),...,f_k(n)\big)$ is totally-homogeneously equidistributed mod 1,
        \item $f_1$ is of the form $f_1(n)=\alpha n+\beta$ for some $\alpha,\beta\in\R$ with $\alpha$ irrational.
    \end{enumerate}
    Then, we have
    \begin{equation*}\label{S1eq:thmdensity}
        \dens\big(\{n\in\N:\gcd(n,\floor{f_1(n)},...,\floor{f_k(n)})=1\}\big)=\frac{1}{\zeta(k+1)}.
    \end{equation*}
\end{theorem}
For example, Theorem \ref{S1thm:totallyhomog} can be applied to show that $\gcd(n,\floor{\pi n}, \floor{\sqrt n})=1$ happens with probability $\zeta(3)^{-1}$, using only the equidistributive properties of the sequence $(\pi n, \sqrt n)$ mod 1.\\

In the case where $f_1,...,f_k$ are polynomials, one may use Weyl's Equidistribution Theorem to show that the sequence $\big(f_1(n),...,f_k(n)\big)$ is equidistributed mod 1 if and only if the polynomial
\begin{equation}\label{S1eq:defineFhx}
    F_\bfh(x):=h_1f_1(x)+\cdots+h_kf_k(x)
\end{equation}
admits an irrational coefficient other than its constant term for any choice of non-zero $\bfh=(h_1,...,h_k)\in\Z^k$ (see Theorems 3.2 and 6.2 from \cite[Chapter 1]{KNunifdistr}).  When considering totally-homogeneous equidistribution, it suffices to notice that
$$h_1\left(\frac{1}{D}f_1(Dx+b)\right)+\cdots+h_k\left(\frac{1}{D}f_k(Dx+b)\right)=\frac{1}{D}F_\bfh(Dx+b)\quad\quad (D,b\in\N),$$
is a polynomial in $x$ which has a non-constant irrational coefficient whenever $F_\bfh(x)$ does. This leads to the following corollary of Theorem \ref{S1thm:totallyhomog}.

\begin{corollary}\label{S1cor:polynomials}
    Let $f_1(x),...,f_k(x)\in\R[x]$ be such that $\deg f_1=1$ and such that the polynomial $F_\bfh(x)$ given by \eqref{S1eq:defineFhx} admits an irrational coefficient other than its constant term for any non-zero $\bfh\in\Z^k$. Then we have
    \begin{equation*}\label{S1eq:polynomials}
        \densityofnwith{\gcd(n,\floor{f_1(n)},...,\floor{f_k(n)})=1}=\frac{1}{\zeta(k+1)}.
    \end{equation*}
\end{corollary}

Corollary \ref{S1cor:polynomials} directly answers Question \ref{S1q:BRmainquestion} when it is applied to the functions 
$$f_1(x)=\alpha_1x,\quad f_2(x)=\alpha_2x^2,\quad ...,\quad f_k(x)=\alpha_kx^k$$ 
with $\alpha_1,...,\alpha_k$ irrational. Indeed, the polynomial $F_\bfh(x)$ then corresponds simply to $h_1\alpha_1x+\cdots+h_k\alpha_k x^k$, which admits the irrational coefficient $h_j\alpha_j$ whenever $h_j$ is a non-zero component of $\bfh=(h_1,...,h_k)\in\Z^k$. Using Corollary \ref{S1cor:polynomials}, one can also prove other density results such as
$$\densityofnwith{\gcd(n,\floor{\alpha_1n},...,\floor{\alpha_kn})=1}=\frac{1}{\zeta(k+1)}$$
when $1,\alpha_1,...,\alpha_k\in\R$ are linearly independent over $\Q$, or
$$\densityofnwith{\gcd(n,\floor{\alpha n},\floor{\beta_1n^2},\floor{\beta_2n^2})=1}=\frac{1}{\zeta(4)}\approx 92.39\%$$
when $\alpha\in\R$ is irrational and $1,\beta_1,\beta_2\in\R$ are linearly independent over $\Q$. In fact, our method could just as easily be used to prove that
$$\densityofnwith{\gcd(\floor{\alpha_1n},...,\floor{\alpha_kn})=1}=\frac{1}{\zeta(k)}$$
whenever $1,\alpha_1,...,\alpha_k\in\R$ are linearly independent over $\Q$, which is a result that was mentioned in passing by Bergelson and Richter (see Theorem 22 from \cite{BRgcdintparthardyfield}).\\ 

It is important to note that Condition (2) from Theorem \ref{S1thm:totallyhomog} (or the assumption $\deg f_1=1$ in Corollary \ref{S1cor:polynomials}) is required only for the purpose of bounding the possible size of $\gcd(n,\floor{f_1(n)})$. However, there are no immediate reason to think that the theorem should not hold with other types of functions. For example, the work presented in \cite{ELprobathatngncoprime} is sufficient to show that Theorem \ref{S1thm:totallyhomog} holds whenever $f_1(x)$ satisfies conditions EL1,... EL5 (in place of Condition 2). Similarly, Theorem 2 in the article of Delmer and Deshouillers \cite{DDpiatetskishapiro} is enough to show that Theorem \ref{S1thm:totallyhomog} holds when $f_1$ is instead of the form $f_1(x)=x^c$ with $c\in(1,+\infty)\setminus\N$. With some more effort, it seems plausible that \cite{BRgcdintparthardyfield} can be used to show the same thing about functions $f_1$ satisfying BR1, BR2 and BR3. \\

In contrast, the case where $f_1(x)$ is a polynomial of degree at least $2$ remains seemingly inaccessible to the current technology. To obtain such results, a first step would be to answer the following, which we leave as an open question.

\begin{question}
    Given an irrational $\alpha\in\R$ and some $k\in\N$ with $k\geq 2$, is it true that
    $$\densityofnwith{\gcd(n,\floor{\alpha n^k})=1}=\frac{1}{\zeta(2)}\,?$$
\end{question}

\noindent\textbf{Acknowledgement.} I would like to thank my Ph.D. supervisor, Professor Liu, who introduced me to the result of Watson and highlighted the possible role of equidistribution. I also want to thank Professor Deshouillers for a very interesting discussion following his talk at the SCHOLAR II conference\footnote{held in 2024 at the Fields Institute in Toronto, CA.}, where he mentioned the open question in the article of Bergelson and Richter. This research was supported in part by the Centre de recherche math\'ematique (CRM).

\section{Proof of Theorem \ref{S1thm:totallyhomog}}\label{S2}

An important observation about the floor function is that $\floor{x}\equiv a$ (mod $D$) occurs only when $x$ is of the form $x=Dm+a+t$ with $m\in\Z, t\in[0,1)$, which is equivalent to $a/D\leq\{x/D\}<(a+1)/D$. Under the assumption that $\{x/D\}$ is uniformly distributed in the interval $[0,1)$, we find that $\floor{x}\equiv a$ (mod $D$) happens with probability $1/D$. Applying this logic to totally-homogeneously equidistributed sequences, we recover the following.

\begin{proposition}\label{S2prop:arithmeticequid}
    Let $k\in\N$ and let $\bfx_n:=\big(x_1(n),...,x_k(n)\big)$ denote a sequence in $\R^k$ which is totally-homogeneously equidistributed mod 1. Then, the sequence $(\bfy_n)_{n\in\N}$ in $\Z^{k+1}$ given by
    \begin{equation*}
        \bfy_n:=(n,\floor{x_1(n)},...\floor{x_k(n)})
    \end{equation*}
    has the property that, for all $D\in\N$ and $\bfc\in\Z^{k+1}$ we have
    \begin{equation*}
        \densityofnwith{\bfy_n\equiv\bfc\textnormal{ (mod }D)}=\frac{1}{D^{k+1}}.
    \end{equation*}
\end{proposition}
\begin{proof}
    Let $\bfc:=(c_0,c_1,...,c_k)$ and, without loss of generality, assume that $c_0\in\{1,...,D\}$ and $c_i\in\{0,...,D-1\}$ for $i=1,...,k$. In particular, $\bfy_n\equiv \bfc$ (mod $D$) implies that $n$ is of the form $n=Dm+c_0$ for some $m\in\N\cup\{0\}$. Hence, to count those $n\in\{1,...,N\}$ with $\bfy_n\equiv\bfc$ (mod $D$), we may simply count those $m\in\{0,...,\floor{(N-c_0)/D}\}$ with
    $$\fracpart{\frac{x_1(Dm+c_0)}{D}}\in\left[\frac{c_1}{D},\frac{c_1+1}{D}\right),\quad ...,\quad \fracpart{\frac{x_k(Dm+c_0)}{D}}\in\left[\frac{c_k}{D},\frac{c_k+1}{D}\right).$$
    Letting $\Bcal:=[c_1/D, (c_1+1)/D)\times\cdots\times [c_k/D, (c_k+1)/D)$ and using the fact that $\big(\frac{1}{D}\bfx_{Dm+c_0}\big)_{m\in\N}$ is equidistributed mod 1, we find
    $$\lim_{N\to\infty}\frac{\left|\left\{m\in\{0,...,\floor{(N-c_0)/D}\}:\{\frac{1}{D}\bfx_{Dm+c_0}\}\in\Bcal\right\}\right|}{\floor{(N-c_0)/D}}=\frac{1}{D^k},$$
    and we deduce that 
    $$\densityofnwith{\bfy_n\equiv\bfc\textnormal{ (mod }D)}=\frac{1}{D^k}\left(\lim_{N\to\infty}\frac{\floor{(N-c_0)/D}}{N}\right)=\frac{1}{D^{k+1}}.$$
    This concludes the proof.
\end{proof}

Given a set $\Acal\subseteq\N$, we may define its \textit{upper} and \textit{lower densities} as
\begin{equation*}
    \upperdens\,\Acal:=\limsup_{N\to\infty}\frac{|\Acal\cap\{1,...,N\}|}{N},\quad \lowerdens\,\Acal:=\liminf_{N\to\infty}\frac{|\Acal\cap\{1,...,N\}|}{N},
\end{equation*}
respectively. Note that $\upperdens\,\Acal$ and $\lowerdens\,\Acal$ are well defined for any $\Acal\subseteq\N$, and $\dens\,\Acal$ exists if and only if $\upperdens\,\Acal=\lowerdens\,\Acal$. We make the following observations regarding descending chains of sets and their natural densities.
\begin{proposition}\label{S2prop:densityofchains}
    Let $\big(\Acal_z\big)_{z\in\N}$ be a sequence of subsets of $\N$ such that
    \begin{enumerate}
        \item $\Acal_z\supseteq \Acal_{z+1}$ for all $z\in\N$,
        \item $\dens\,\Acal_z$ exists for all $z\in\N$,
        \item the set $\Acal:=\cap_{z\in\N}\Acal_z$ satisfies
        $$\lim_{z\to\infty}\upperdens(\Acal_z\setminus\Acal)=0.$$
    \end{enumerate}
    Then, $\dens\,\Acal$ exists and is equal to $\displaystyle \lim_{z\to\infty}\dens\,\Acal_z$.
\end{proposition}
\begin{proof}
    Note that $(\dens\,\Acal_z)_{z\in\N}$ is a decreasing sequence in $[0,1]$, so it necessarily converges.
    Since $\Acal\subseteq\Acal_z$, we necessarily have 
    $$\upperdens\,\Acal\leq \upperdens\,\Acal_z=\dens\,\Acal_z$$
    for all $z\in\N$, and taking the limit as $z\to\infty$ gives $\upperdens\,\Acal\leq \displaystyle \lim_{z\to\infty}\dens(\Acal_z)$. Similarly, if we apply the general inequality
    $$\lowerdens(\Bcal\setminus\Ccal)\geq \lowerdens\,\Bcal-\upperdens\,\Ccal\quad(\Ccal\subseteq\Bcal\subseteq\N)$$
    to $\Acal=\Acal_z\setminus(\Acal_z\setminus\Acal)$, we obtain
    $$\lowerdens\,\Acal\geq \lowerdens\,\Acal_z-\upperdens(\Acal_z\setminus\Acal)=\dens\,\Acal_z-\upperdens(\Acal_z\setminus\Acal).$$
    Again taking the limit as $z\to\infty$, we obtain $\lowerdens\,\Acal\geq \displaystyle \lim_{z\to\infty}\dens\,\Acal_z$, which concludes the proof.
\end{proof}

Using this principle, we may reduce the proof of Theorem \ref{S1thm:totallyhomog} to that of the following lemma.

\begin{lemma}\label{S2lem:upperdensityvanish}
    For any $\alpha,\beta\in\R$ with $\alpha$ irrational, we have
    $$\lim_{z\to\infty}\upperdens\big(\{n\in\N:\exists\textnormal{ a prime }p>z \textnormal{ such that }p\mid\gcd(n,\floor{\alpha n+\beta})\}\big)=0.$$
\end{lemma}

Lemma \ref{S2lem:upperdensityvanish} is proven more generally in Section \ref{S4} (see Lemma \ref{S4lem:upperdens}). First, we quickly show how Lemma \ref{S2lem:upperdensityvanish} directly implies our main result.

\begin{proof}[Proof of Theorem \ref{S1thm:totallyhomog}]
    We define
    \begin{align*}
        \Acal_z&:=\{n\in\N:\forall\textnormal{ prime }p\leq z, \textnormal{ we have } p\nmid\gcd(n,\floor{f_1(n)},...,\floor{f_k(n)})\}
    \end{align*}
    for each $z\in\N$. Let $D_z:=\prod_{p\leq z}p$ and consider the set 
    $$\Ccal_z:=\{\bfc\in\{0,...,D_z-1\}^{k+1}:\forall \textnormal{ prime }p\leq z, \quad\bfc\not\equiv {\bf 0}\textnormal{ (mod }p)\},$$
    which, by the Chinese Remainder Theorem, has cardinality
    $$|\Ccal_z|=\prod_{p\mid D_z}|\{\bfc\in\{0,...,p-1\}^{k+1}:\bfc\neq{\bf 0}\}|=\prod_{p\leq z}\big(p^{k+1}-1\big).$$
    Letting $\bfy_n:=\big(n,\floor{f_1(n)},...\floor{f_k(n)}\big)$, we find
    $$\Acal_z=\{n\in\N:\forall\textnormal{ prime }p\leq z, \quad\bfy_n\not\equiv {\bf 0}\textnormal{ (mod }p)\}=\bigcup_{\bfc\in\Ccal_z}\{n\in\N:\bfy_n\equiv \bfc \textnormal{ (mod }D_z)\}$$
    where this last union is disjoint, and it follows from Proposition \ref{S2prop:arithmeticequid} that
    $$\dens\,\Acal_z=\sum_{\bfc \in\Ccal_z}\dens\big(\{n\in\N:\bfy_n\equiv\bfc\textnormal{ (mod }D_z)\}\big)=\frac{|\Ccal_z|}{D_z^{k+1}}=\prod_{p\leq z}\left(1-\frac{1}{p^{k+1}}\right).$$
    Now, observe that the sets $\Acal_z$ $(z\in\N)$ form a descending chain with
    $$\Acal:=\bigcap_{z\in\N}\Acal_z=\left\{n\in\N:\gcd\big(n,\floor{f_1(n)},...,\floor{f_k(n)}\big)=1\right\},$$
    and also they satisfy
    \begin{align*}
        \Acal_z\setminus\Acal\subseteq&\{n\in\N:\exists\textnormal{ a prime }p>z \textnormal{ such that }p\mid\gcd(n,\floor{f_1(n)},...,\floor{f_k(n)})\}\\
        \subseteq&\{n\in\N:\exists\textnormal{ a prime }p>z \textnormal{ such that }p\mid\gcd(n,\floor{f_1(n)})\}.
    \end{align*}
    Knowing that $f_1(n)=\alpha n+\beta$ for some $\alpha,\beta\in\R$ with $\alpha$ irrational, Lemma \ref{S2lem:upperdensityvanish} guarantees that
    $$\lim_{z\to\infty}\upperdens\big(\Acal_z\setminus\Acal\big)=0.$$
    By Proposition \ref{S2prop:densityofchains}, we find that $\dens\,\Acal$ exists and is equal to 
    $$\dens\,\Acal=\lim_{z\to\infty}\dens\,\Acal_z=\lim_{z\to\infty}\prod_{p\leq z}\left(1-\frac{1}{p^{k+1}}\right)=\frac{1}{\zeta(k+1)},$$
    which concludes the proof.
\end{proof}

\section{Counting lattice points in convex regions of the plane}\label{S3}

For the sake of Section \ref{S4}, we obtain some preliminary bounds on the number of lattice points contained in a convex region of $\R^2$ depending on its area. In particulat, working in dimension $2$ allows for the use of Pick's Theorem, which we recall below.

\begin{theorem}[Pick's Theorem]\label{S3thm:Picks}
    Let $\Pcal\subseteq\R^2$ be a closed polygonal region of $\R^2$ whose area is non-zero and whose vertices lie on $\Z^2$. Then
    $$\Area(\Pcal)=P_0+\frac{1}{2}P_1-1$$
    where $P_0$ is the number of integer points lying in the interior of $\Pcal$, and $P_1$ is the number of integer points lying on the boundary of $\Pcal$.
\end{theorem}

In our case, we are particularly interested by the following corollary of Pick's Theorem.

\begin{corollary}\label{S3cor:ptinconvex}
    Let $\Ccal$ be a bounded convex set in $\R^2$ and let $\Lambda$ be a lattice in $\R^2$. Then, either $\Ccal\cap\Lambda$ is entirely contained in an affine line, or
    $$|\Ccal\cap\Lambda|\leq 2\,\frac{\Area(\Ccal)}{\det\Lambda}+2.$$
\end{corollary}
\begin{proof}
    We first prove the case $\Lambda=\Z^2$. Let $\Pcal\subseteq\R^2$ denote the convex hull of $\Ccal\cap\Z^2$, that is
    $$\Pcal=\left\{\sum_{\bfx\in\Ccal\cap\Z^2}t_{\bfx}\bfx\in\R^2:t_{\bfx}\in[0,1]\quad (\bfx\in\Ccal\cap\Z^2),\quad \sum_{\bfx\in\Ccal\cap\Z^2}t_{\bfx}= 1\right\}.$$
    Assuming that $\Ccal\cap\Z^2$ is not contained in a line, we may find three points $\bfx_1,\bfx_2,\bfx_3\in\Ccal\cap\Z^2$ which are not collinear. Since $\Pcal$ is convex, it contains the triangle formed by $\bfx_1,\bfx_2,\bfx_3$ which has a non-zero area, which in turn gives $\Area(\Pcal)>0$. Moreover, since $\Ccal\cap\Z^2$ is finite, we find that $\Pcal$ is a closed polygonal region of $\R^2$ with vertices taken from $\Ccal\cap\Z^2$.\\
    \begin{figure}[hbt!]
        \centering
        \label{fig:constructionofP}
        \begin{tikzpicture}
            \draw[white] (0,5) rectangle (14,0);

            \draw[color=black] (3.5,2.5) circle [x radius=2, y radius=1, rotate=30];
            \draw (5.5,4) node {$\Ccal$};

            \foreach \x in {0,...,10}{
                \foreach \y in {0,...,8}
            \draw (1.1+0.5*\x,0.3+0.5*\y) node {$\cdot$};
            }

            \draw (7, 2.5) node {\Large$\longrightarrow$};

            \draw[color=black] (10.5,2.5) circle [x radius=2, y radius=1, rotate=30];
            \draw (12.5,4) node {$\Ccal$};

            \draw (9.1,1.3) -- (9.1, 2.3);
            \draw (9.1,2.3) -- (10.1, 3.3);
            \draw (10.1, 3.3) -- (11.6, 3.8);
            \draw (11.6, 3.8) -- (12.1, 3.3);
            \draw (12.1, 3.3) -- (12.1, 2.8);
            \draw (12.1, 2.8) -- (11.1,1.8);
            \draw (11.1,1.8) -- (10.1,1.3);
            \draw (10.1, 1.3) -- (9.1, 1.3);
            
            \draw (11.8,3) node {$\Pcal$};
            
            \draw[red] (9.1,1.3) node {$\cdot$};
            \draw[red] (9.1,1.8) node {$\cdot$};
            \draw[red] (9.1,2.3) node {$\cdot$};
            
            \draw[red] (9.6,1.3) node {$\cdot$};
            \draw[red] (9.6,1.8) node {$\cdot$};
            \draw[red] (9.6,2.3) node {$\cdot$};
            \draw[red] (9.6,2.8) node {$\cdot$};
            
            \draw[red] (10.1,1.3) node {$\cdot$};
            \draw[red] (10.1,1.8) node {$\cdot$};
            \draw[red] (10.1,2.3) node {$\cdot$};
            \draw[red] (10.1,2.8) node {$\cdot$};
            \draw[red] (10.1,3.3) node {$\cdot$};
            
            \draw[red] (10.6,1.8) node {$\cdot$};
            \draw[red] (10.6,2.3) node {$\cdot$};
            \draw[red] (10.6,2.8) node {$\cdot$};
            \draw[red] (10.6,3.3) node {$\cdot$};
            
            \draw[red] (11.1,1.8) node {$\cdot$};
            \draw[red] (11.1,2.3) node {$\cdot$};
            \draw[red] (11.1,2.8) node {$\cdot$};
            \draw[red] (11.1,3.3) node {$\cdot$};
            
            \draw[red] (11.6,2.3) node {$\cdot$};
            \draw[red] (11.6,2.8) node {$\cdot$};
            \draw[red] (11.6,3.3) node {$\cdot$};
            \draw[red] (11.6,3.8) node {$\cdot$};
            
            \draw[red] (12.1,2.8) node {$\cdot$};
            \draw[red] (12.1,3.3) node {$\cdot$};
                
        \end{tikzpicture}.
        \caption{Construction of $\Pcal$.}
    \end{figure}
    
    Since $\Ccal$ is convex, we have $\Pcal\subseteq\Ccal$ and, applying Pick's Theorem (Theorem \ref{S3thm:Picks}), we find
    $$\Area(\Ccal)\geq \Area(\Pcal)=P_0+\frac{P_1}{2}-1\geq \frac{P_0+P_1}{2}-1=\frac{|\Ccal\cap\Z^2|}{2}-1$$
    where $P_0$ and $P_1$ are defined as in Theorem $\ref{S3thm:Picks}$, and the equality $P_0+P_1=|\Ccal\cap\Z^2|$ is a direct consequence of the construction of $\Pcal$. Rearranging this last inequality gives $|\Ccal\cap\Z^2|\leq 2\,\Area(\Ccal)+2$, as desired.\\

    We now prove the general case $\Lambda=T\Z^2$ where $T$ is some arbitrary $2\times2$ invertible matrix with real coefficients. Letting $\Ccal':=T^{-1}\Ccal$, we find that $T$ yields a bijection between $\Ccal'\cap\Z^2$ and $\Ccal\cap\Lambda$. If $\Ccal'\cap\Z^2$ is contained in some affine line $\Lcal\subseteq\R^2$, then $\Ccal\cap\Lambda$ is contained in the affine line $T\Lcal$ and we are done. Otherwise, since $\Ccal'$ is convex, we have
    $$|\Ccal\cap\Lambda|=|\Ccal'\cap\Z^2|\leq 2\,\Area(\Ccal')+2=2\,\frac{\Area(\Ccal)}{|\det T|}+2.$$
    This concludes the proof since $\det \Lambda:=|\det T|.$
\end{proof}

We end this section by considering primitive lattice points. Given a lattice $\Lambda\subseteq\R^m$ $(m\geq1)$, we say that a point $\bfx\in\Lambda$ is \textit{primitive} in $\Lambda$ if $\bfx\neq0$ and the intersection between $\Lambda$ and the line generated by $\bfx$ is itself generated by $\bfx$ over $\Z$, which is to say that $\R\bfx\cap\Lambda=\Z\bfx$. We denote by $\Lambda^*$ the set of all primitive points in $\Lambda$. For example, if we take the lattice of integer points $\Lambda=\Z^m$ in $\R^m$, then $\Lambda^*$ contains exactly the points $(x_1,...,x_m)\in\Z^m$ with $\gcd(x_1,...,x_m)=1$.\\ 

When the convex region $\Ccal$ is symmetric and we only count the points of $\Ccal\cap\Lambda$ which are primitive, we do not have to account for the possibility that $\Ccal\cap\Lambda$ is contained in an affine line.

\begin{corollary}\label{S3cor:primitiveptinconvex}
    Let $\Ccal$ be a bounded symmetric convex set in $\R^2$ and let $\Lambda$ be a lattice in $\R^2$. Then, 
    $$|\Ccal\cap\Lambda^*|\leq 2\,\frac{\Area(\Ccal)}{\det\Lambda}+2.$$
\end{corollary}
\begin{proof}
    First, we suppose that $\Ccal\cap\Lambda\subseteq\Lcal$ for some affine line $\Lcal\subseteq\R^2$. If $\Ccal=\varnothing$ we are done, and otherwise we necessarily have $0\in\Ccal$. Note that for all $\bfx\in\Ccal\cap\Lambda^*$, we have $0,\bfx\in\Lcal$, which gives $\R\bfx=\Lcal$ and
    $$\Lcal\cap\Lambda=\R\bfx\cap\Lambda=\Z\bfx.$$
    Since the group $\Z\bfx$ has exactly two generators, $\bfx$ and $-\bfx$, we find $|\Ccal\cap\Lambda^*|=2$. If we suppose instead that $\Ccal\cap\Lambda$ is not contained in a line, then Corollary \ref{S3cor:ptinconvex} gives
    $$|\Ccal\cap\Lambda^*|\leq |\Ccal\cap\Lambda|\leq 2\,\frac{\Area(\Ccal)}{\det\Lambda}+2,$$
    and we are done.
\end{proof}

\section{On the greatest common divisor of two Beatty sequences}\label{S4}
This section is dedicated to the proof of Lemma \ref{S2lem:upperdensityvanish}, where we follow a similar method to that of Estermann \cite{ESTERMANNprimitivelatticepoints}. We can however improve the quality of certain estimates and streamline the argument through the use of Pick's Theorem (in the form of Corollary \ref{S3cor:primitiveptinconvex}).\\ 

For the rest of this section, we fix $\alpha_1,\alpha_2,\beta_1,\beta_2\in\R$ and we consider the corresponding (inhomogeneous) Beatty sequences given by
$$B_i(n):=\floor{\alpha_in+\beta_i}\quad\quad (i=1,2).$$
Also, we write $X\ll Y$ to mean that there exists a constant $C>0$ depending at most on $\alpha_1,\alpha_2,\beta_1,\beta_2$ such that $|X|\leq C Y$.

\begin{lemma}\label{S4lem:gcdequald}
    Suppose that $\alpha_1\neq0$. Then, for all $d,N\in\N$, the set
    \begin{equation*}
        \Gcal_d(N):=\{n\in\{1,...,N\}:\gcd\big(B_1(n),B_2(n)\big)=d\}
    \end{equation*}
    satisfies $|\Gcal_d(N)|\ll N/d^2+1.$
\end{lemma}
\begin{proof}
    We observe that, for all $n\in\{1,...,N\}$ we have
    \begin{align*}
        |B_1(n)|&\leq |\alpha_1|N+|\beta_1|+1,\\
        |\alpha_2B_1(n)-\alpha_1B_2(n)|&\leq|\alpha_2\beta_1-\alpha_1\beta_2|+|\alpha_1|+|\alpha_2|.
    \end{align*}
    Accordingly, we consider the parallelogram $\Ccal\subseteq\R^2$ formed by all points $(x,y)\in\R^2$ satisfying the inequalities 
    $$|x|\leq |\alpha_1|N+|\beta_1|+1,\quad |\alpha_2x-\alpha_1y|\leq |\alpha_2\beta_1-\alpha_1\beta_2|+|\alpha_1|+|\alpha_2|$$
    which has area equal to
    $$\Area(\Ccal)=\frac{4}{|\alpha_1|}\big(|\alpha_1|N+|\beta_1|+1\big)\big(|\alpha_2\beta_1-\alpha_1\beta_2|+|\alpha_1|+|\alpha_2|\big)\ll N.$$
    Applying Corollary \ref{S3cor:primitiveptinconvex} with the lattice $\Lambda:=d\Z\times d\Z$, we obtain
    $$|\{(x,y)\in\Ccal\cap\Z^2:\gcd(x,y)=d\}|=|\Ccal\cap\Lambda^*|\leq 2\,\frac{\Area(\Ccal)}{\det\Lambda}+2\ll \frac{N}{d^2}+1.$$
    To conclude the proof, note that $n\in\Gcal_d(N)$ only occurs when $\big(B_1(n),B_2(n)\big)\in\Ccal\cap\Lambda^*$. Since the equality $B_1(n)=\floor{\alpha_1n+\beta_1}=x$ has at most $|\alpha_1|^{-1}+1$ solutions $n\in\Z$ for any $x\in\Z$, we find $|\Gcal_d(N)|\leq \big(|\alpha_1|^{-1}+1\big)|\Ccal\cap\Lambda^*|\ll N/d^2+1$.
\end{proof}

\begin{remark}
    Our Lemma \ref{S4lem:gcdequald} effectively plays the same role as Lemma 2 of Estermann \cite{ESTERMANNprimitivelatticepoints}. However, using Pick's Theorem allows one to remove a term of order $\log N$ in the upper bound for $|\Gcal_d(N)|$.
\end{remark}

When $\alpha_1\neq0$, Dirichlet's Approximation Theorem guarantees that, for any $N\in\N$, there exist $q\in\N,a\in\Z$ with
\begin{equation}\label{eq:Dirichlet}
    1\leq q\leq N^{1/2},\quad \left|q(\alpha_2/\alpha_1)-a\right|\leq N^{-1/2},\quad \gcd(a,q)=1.
\end{equation}
For each $N\in\N$, we define $q(N):=q(\alpha_1,\alpha_2,N)\in\N$ as the largest integer $q$ for which \eqref{eq:Dirichlet} admits a solution $a\in\Z$. We quickly observe that $\big(q(N)\big)_{N\in\N}$ is a non-decreasing sequence and that $\displaystyle\lim_{N\to\infty}q(N)=\infty$ happens if and only if $\alpha_2/\alpha_1$ is irrational.
\begin{lemma}\label{S4lem:ddividesgcd}
    Suppose that $\alpha_1\neq0$. Then, for all $d,N\in\N$, we have
    $$|\{n\in\{1,...,N\}: q(N)\nmid B_1(n),\hfspace d\mid B_1(n), \hfspace d\mid B_2(n)\}|\ll\frac{N}{d^2}+\frac{N^{1/2}}{d}.$$
\end{lemma}
\begin{proof}
    Write $q=q(N)$ for simplicity, and let $a\in\Z$ be an integer satisfying $\gcd\big(a,q\big)=1$ and $|q(\alpha_2/\alpha_1)-a|\leq N^{-1/2}$.  When $n\in\N$ is such that $q \nmid B_1(n)$, then $aB_1(n)\not\equiv0$ mod $q $, which gives $qB_2(n)-aB_1(n)\neq0$. In that case, we have
    \begin{align*}
        \gcd\big(B_1(n),B_2(n)\big)\leq&\,|qB_2(n)-aB_1(n)|\\
        \leq&\,\big|n(q\alpha_2-a\alpha_1)\big|+|q|+|a|\\
        \ll&\, N\cdot N^{-1/2}+N^{1/2}\\
        \ll&\, N^{1/2},
    \end{align*}
    where we used the fact that $|a|\leq |q(\alpha_2/\alpha_1)|+N^{-1/2}\ll N^{1/2}$. Hence, we have shown that there exists a constant $C>0$ depending at most on $\alpha_1,\alpha_2,\beta_1,\beta_2$ such that, when $q\nmid B_1(n)$ $(n\in\{1,...,N\})$, we have $\gcd\big(B_1(n),B_2(n)\big)\leq CN^{1/2}$. In other words, this means that $\gcd\big(B_1(n),B_2(n)\big)=dr$ for some integer $r\leq (CN^{1/2})/d$, and using Lemma \ref{S4lem:gcdequald} we find
    \begin{align*}
        |\{n\in\{1,...,N\}: q\nmid B_1(n),\hfspace d\mid B_1(n), \hfspace d\mid B_2(n)\}|
        \leq& \sum_{r\leq\frac{CN^{1/2}}{d}}|\Gcal_{dr}(N)|\\
        \ll& \sum_{r\leq\frac{CN^{1/2}}{d}}\left(\frac{N}{(dr)^2}+1\right)\\
        \ll& \frac{N}{d^2}+\frac{N^{1/2}}{d}.
    \end{align*}
    This concludes the proof.
\end{proof}

We are now ready to prove Lemma \ref{S2lem:upperdensityvanish} in a slightly more explicit form. Note that Lemma \ref{S2lem:upperdensityvanish} corresponds to the case $\alpha_1=1,\beta_1=0$ and $\alpha_2\notin\Q$.

\begin{lemma}\label{S4lem:upperdens}
    If $\alpha_1,\alpha_2$ are linearly independent over $\Q$, then
    $$\upperdens\Big(\{n\in\N:\exists\textnormal{ a prime }p>z \textnormal{ such that }p\mid B_1(n),\hfspace p\mid B_2(n)\}\Big)\ll \sum_{p>z}\frac{1}{p^2}$$
    for all $z\in\N$.
\end{lemma}
\begin{proof}
    We define $P(N)=\max\big\{|B_1(n)|:n\in\{1,...,N\}\big\}$ which has size $P(N)\ll N$. First, observe that
    \begin{align*}
        |\{n\in\{1,...,N\}:q(N)\mid B_1(n)\}|=&\sum_{|r|\leq \frac{P(N)}{q(N)}}|\{n\in\{1,...,N\}:B_1(n)=q(N)r\}|\\
        \leq&\left(2\,\frac{P(N)}{q(N)}+1\right)\left(\frac{1}{|\alpha_1|}+1\right)\\
        \ll& \frac{N}{q(N)}.
    \end{align*}
    
    Applying Lemma \ref{S4lem:ddividesgcd}, we get
    \begin{align*}
        &|\{n\in\{1,...,N\}:\exists\textnormal{ a prime }p>z \textnormal{ such that }p\mid B_1(n),\hfspace p\mid B_2(n)\}|\\
        \leq&|\{n\in\{1,...,N\}:q(N)\mid B_1(n)\}|\\
        &+\sum_{z<p\leq P(N)}|\{n\in\{1,...N\}:q(N)\nmid B_1(n),\hfspace p\mid B_1(n),\hfspace p\mid B_2(n)\}|\\
        \ll& \frac{N}{q(N)}+\sum_{z<p\leq P(N)}\left(\frac{N}{p^2}+\frac{N^{1/2}}{p}\right)\\
        \ll& \frac{N}{q(N)}+N\left(\sum_{p>z}\frac{1}{p^2}\right)+N^{1/2}\log\log P(N).
    \end{align*}
    To obtain the desired bound, it suffices to observe that
    $$\limsup_{N\to\infty}\frac{1}{N}\left(\frac{N}{q(N)}+N^{1/2}\log\log P(N)\right)=\limsup_{N\to\infty}\left(\frac{1}{q(N)}+\frac{\log\log P(N)}{N^{1/2}}\right)=0,$$
    where we used $q(N)\to\infty$ and $P(N)\ll N$.
\end{proof}

\end{document}